\documentclass[graybox]{svmult}

\usepackage{type1cm}        
\usepackage{makeidx}         
\usepackage{graphicx}        
\usepackage{multicol}        
\usepackage[bottom]{footmisc}

\usepackage{newtxtext}       %
\usepackage{newtxmath}       
\usepackage{multirow}

\usepackage[section]{algorithm}
\usepackage{algorithmic}

\makeindex             

\DeclareMathOperator{\supp}{supp}
\let\Re\relax
\DeclareMathOperator{\Re}{Re}
\DeclareMathOperator{\polylog}{polylog}

\begin{document}

\title*{Sparse approximation of multivariate functions from small datasets via weighted orthogonal matching pursuit}
\titlerunning{Sparse approximation of multivariate functions from small datasets via weighted OMP}
\author{Ben Adcock and Simone Brugiapaglia}
\institute{Ben Adcock \at Simon Fraser University, 
Burnaby, BC, Canada, \email{ben\_adcock@sfu.ca}
\and Simone Brugiapaglia \at Simon Fraser University, 
Burnaby, BC, Canada, \email{simone\_brugiapaglia@sfu.ca}}
%
%
\maketitle

\abstract*{We show the potential of greedy recovery strategies for the sparse approximation of multivariate functions from a small dataset of pointwise evaluations by considering an extension of the orthogonal matching pursuit to the setting of weighted sparsity. The proposed recovery strategy is based on a formal derivation of the greedy index selection rule. Numerical experiments show that the proposed weighted orthogonal matching pursuit algorithm is able to reach accuracy levels similar to those of weighted $\ell^1$ minimization programs while considerably improving the computational efficiency for small values of the sparsity level.}

\abstract{We show the potential of greedy recovery strategies for the sparse approximation of multivariate functions from a small dataset of pointwise evaluations by considering an extension of the orthogonal matching pursuit to the setting of weighted sparsity. The proposed recovery strategy is based on a formal derivation of the greedy index selection rule. Numerical experiments show that the proposed weighted orthogonal matching pursuit algorithm is able to reach accuracy levels similar to those of weighted $\ell^1$ minimization programs while considerably improving the computational efficiency for small values of the sparsity level.}

\section{Introduction}

In recent years, a new class of approximation strategies based on compressive sensing (CS) has been shown to be able to substantially lessen the curse of dimensionality in the context of approximation of multivariate functions from pointwise data, with applications to the uncertainty quantification of partial differential equations with random inputs. Based on random sampling from orthogonal polynomial systems and on weighted $\ell^1$ minimization, these techniques are able to accurately recover a sparse approximation to a function of interest from a small-sized datasets of pointwise samples. In this paper, we show the potential of weighted greedy techniques as an alternative to convex minimization programs based on weighted $\ell^1$ minimization in this context. 

The contribution of this paper is twofold. First, we propose a weighted orthogonal matching pursuit (WOMP) algorithm  based on a rigorous derivation of the corresponding greedy index selection strategy. Second, we numerically show that WOMP is a promising alternative to convex recovery programs based on weighted $\ell^1$ minimization, thanks to its ability to compute sparse approximations with an accuracy comparable to those computed via weighted $\ell^1$ minimization, but with a considerably lower computational cost when the target sparsity level (and, hence, the number of WOMP iterations) is small enough. It is also worth observing here that WOMP computes approximations that are exactly sparse, as opposed to approaches based on weighted $\ell^1$ minimization, which provide compressible approximations in general.\\

\noindent\textbf{Brief literature review.} Various approaches for multivariate function approximation based on CS with applications to uncertainty quantification can be found in \cite{adcock2018infinite,adcock2017compressed,bouchot2017multi,chkifa2018polynomial,doostan2011non,mathelin2012compressed,peng2014weighted,rauhut2017compressive,yang2013reweighted}. An overview of greedy methods for sparse recovery in CS and, in particular of OMP, can be found in \cite[Chapter 3.2]{foucart2013mathematical}. For a general review on greedy algorithms, we refer the reader to \cite{temlyakov2008greedy} and references therein. Some numerical experiments on a weighted variant of OMP  have been performed in the context of CS methods for uncertainty quantification in \cite{bouchot2017multi}. Weighted variants of OMP have also been considered in \cite{li2013weighted,xiao2013weighted}, but the weighted procedure is tailored for specific signal processing applications and the term ``weighted'' does not refer to the weighted sparsity setting of \cite{rauhut2016interpolation} employed here. To the authors' knowledge, the weighted variant of OMP considered in this paper seems to have been proposed here for the first time.\\

\noindent\textbf{Organization of the paper.} In \S\ref{sec:fun_approx} we describe the setting of sparse multivariate function approximation in orthonormal systems via random sampling and weighted $\ell^1$ minimization. Then, in \S\ref{sec:WOMP} we formally derive a strategy for the greedy selection in the weighted sparsity setting and present the WOMP algorithm. Finally, we numerically show the effectiveness of the proposed technique in \S\ref{sec:numerics} and give our conclusions in \S\ref{sec:conclusions}.

\section{Sparse multivariate function approximation}
\label{sec:fun_approx}

We start by briefly introducing the framework of sparse multivariate function approximation from pointwise samples and refer the reader to \cite{adcock2017compressed} for further details. 

Our aim is to approximate a function defined over a high-dimensional domain
$$
f : D \to \mathbb{C}, \quad \text{with } D=(-1,1)^d,
$$ 
where $d \gg 1$, from a dataset of pointwise samples $f(t_1),\ldots,f(t_m)$. Let $\nu$ be a probability measure on $D$ and let $\{\phi_j\}_{j \in \mathbb{N}_0^d}$ be an orthonormal basis for the Hilbert space $L^2_\nu(D)$. In this paper, we will consider $\{\phi_j\}_{j \in \mathbb{N}_0^d}$ to be a tensorized family of Legendre or Chebyshev orthogonal polynomials, with $\nu$ being the uniform or the Chebyshev measure on $D$, respectively. Assuming that $f \in L^2_\nu(D) \cap L^\infty(D)$, we consider the series expansion
$$
f = \sum_{j \in \mathbb{N}_0^d} x_j \phi_j.
$$
Then, we choose a finite set of multi-indices $\Lambda \subseteq \mathbb{N}_0^d$ with $|\Lambda| = N$ and obtain the truncated series expansion
$$
f_\Lambda = \sum_{j \in \Lambda} x_j \phi_j.
$$
In practice, a convenient choice for $\Lambda$ is the hyperbolic cross of order $s$, i.e.
$$
\Lambda := \left\{j \in \mathbb{N}^d_0 : \prod_{k = 1}^d (j_k+1) \leq s \right\},
$$ 
due to the moderate growth of $N$ with respect to $d$. Now, assuming we collect $m\ll N$ pointwise samples independently distributed according to $\nu$, namely,
$$
f(t_1), \ldots, f(t_m),\quad 
\text{with} \quad 
t_1,\ldots,t_m \stackrel{\text{i.i.d.}}{\sim} \nu,
$$
the approximation problem can be recasted as a linear system
\begin{equation}
\label{eq:lin_sys}
A x_\Lambda = y + e,
\end{equation}
with $x_\Lambda = (x_j)_{j \in \Lambda} \in \mathbb{C}^N$, and where the sensing matrix $A \in \mathbb{C}^{m \times N}$ and the measurement vector $y \in\mathbb{C}^m$ are defined as
\begin{equation}
\label{eq:def_A_y}
A_{ij} := \frac{1}{\sqrt{m}}\phi_j(t_i), 
\quad y_i := \frac{1}{\sqrt{m}}f(t_i), 
\quad \forall i \in[m], \; \forall j \in [N],
\end{equation}
with $[k]:=\{1,\ldots,k\}$ for every $k \in \mathbb{N}$. The vector $e \in \mathbb{C}^m$ accounts for the truncation error introduced by $\Lambda$ and satisfies $\|e\|_2 \leq \eta$, where $\eta >0$ is an \emph{a priori} upper bound to the truncation $L^\infty(D)$-error, namely $\|f-f_{\Lambda}\|_{L^\infty(D)} \leq \eta$. A sparse approximation to the vector  can be then computed by means of weighted $\ell^1$ minimization. 

Given weights $w \in \mathbb{R}^N$ with $w > 0$ (where the inequality is read componentwise), recall that the weighted $\ell^1$ norm of a vector $z \in \mathbb{C}^N$ is defined as $\|z\|_{1,w}:=\sum_{j \in [N]}|z_j|w_j$. We can compute an approximation $\hat{x}_\Lambda$ to $x_\Lambda$ by solving the weighted quadratically-constrained basis pursuit (WQCBP) program
\begin{equation}
\label{eq:WQCBP}
\hat{x}_\Lambda \in \arg\min_{z \in \mathbb{C}^N} \|z\|_{1,w}, 
\quad \text{s.t.} \quad \|A z-y\|_{2} \leq \eta,
\end{equation}
where the weights $w \in \mathbb{R}^N$ are defined as
\begin{equation}
\label{eq:weights}
w_j = \|\phi_j\|_{L^\infty(D)}.
\end{equation}
The effectiveness of this particular choice of $w$ is supported by theoretical results and it has been validated from the numerical viewpoint (see \cite{adcock2018infinite,adcock2017compressed}). The resulting approximation $\hat{f}_\Lambda$ to $f$ is finally defined as
$$
\hat{f}_\Lambda := \sum_{j \in \Lambda} (\hat{x}_\Lambda)_j \phi_j.
$$
In this setting, stable and robust recovery guarantees in high probability can be shown for the approximation errors $\|f-f_\Lambda\|_{L^2_\nu(D)}$ and $\|f-f_\Lambda\|_{L^\infty_\nu(D)}$ under a sufficient condition on the number of samples of the form
$
 m \gtrsim s^\gamma \cdot \polylog(s,d),
$ 
with $\gamma = 2$ or $\gamma = \log(3)/\log(2)$ for tensorized Legendre or Chebyshev polynomials, respectively, hence lessening the curse of dimensionality to a substantial extent (see \cite{adcock2017compressed} and references therein). We also note in passing that decoders such as the weighted LASSO or the weighted square-root LASSO can be considered as alternatives to \eqref{eq:WQCBP} for weighted $\ell^1$ minimization (see \cite{adcock2017correcting}).

\section{Weighted orthogonal matching pursuit}
\label{sec:WOMP}

In this paper, we consider greedy sparse recovery strategies to find sparse approximate solutions to \eqref{eq:lin_sys}, as alternatives to the WQCBP optimization program \eqref{eq:WQCBP}. With this aim, we propose a variation of the OMP algorithm to the weighted setting. 

Before introducing weighted OMP (WOMP) in Algorithm~\ref{alg:WOMP}, let us  recall the rationale behind the greedy index selection rule of OMP (corresponding to Algorithm~\ref{alg:WOMP} with $\lambda = 0$ and $w = 1$). For a detailed introduction to OMP, we refer the reader to \cite[Section 3.2]{foucart2013mathematical}. Given a support set $S\subseteq[N]$, OMP solves the least-squares problem
$$
\min_{z \in \mathbb{C}^N} G_0(z) \text{ s.t. } \supp(z) \subseteq S,
$$
where $G_0(z) :=  \|y-Az\|_2^2$. In OMP, the support $S$ is iteratively enlarged by one index at the time. Namely, we consider the update $S \cup \{j\}$, where the index $j \in [N]$ is selected in a greedy fashion. In particular, assuming that $A$ has $\ell^2$-normalized columns, it is possible to show that (see  \cite[Lemma 3.3]{foucart2013mathematical}) 
\begin{equation}
\label{eq:minG0_t}
\min_{t \in \mathbb{C}} G_0(x + t e_j) = G_0(x) - |(A^*(y-Ax))_j|^2.
\end{equation}
This leads to the greedy index selection rule operated by OMP, which prescribes the selection of an index $j \in [N]$ that maximizes the quantity $|(A^*(y-Ax))_j|^2$. We will use this simple intuition to extend OMP to the weighted case by replacing the function $G_0$ with a suitable function $G_\lambda$ that takes into account the data-fidelity term and the weighted sparsity prior at the same time. 

Let us recall that, given a set of weights $w \in \mathbb{R}^N$ with $w >0$, the weighted $\ell^0$ norm of a vector  $z \in \mathbb{C}^N$ is defined as  the quantity (see \cite{rauhut2016interpolation})\footnote{The term ``norm'' here is an abuse of language, but we will stick to it due to its popularity.}
$$
\|z\|_{0,w} := \sum_{j \in \supp(z)} w_j^2.
$$
Notice that when $w = 1$, then $\|\cdot\|_{0,w} = \|\cdot\|_0$ is the standard $\ell^0$ norm. Given $\lambda  \geq 0$, we define the function
\begin{equation}
\label{eq:defGlambda}
G_\lambda(z) := \|y-Az\|_2^2 + \lambda \|z\|_{0,w}.
\end{equation}
The tradeoff between the data-fidelity constraint and the weighted sparsity prior is balanced via the choice of the regularization parameter $\lambda$. Applying the same rationale employed in OMP for the greedy index selection and replacing $G_0$ with $G_\lambda$ leads to Algorithm~\ref{alg:WOMP}, which corresponds to OMP when $\lambda = 0$ and $w = 1$. 
\begin{algorithm}
\caption{\label{alg:WOMP}Weighted orthogonal matching pursuit (WOMP)}
\textbf{Inputs:}
\begin{itemize}
\item $A\in \mathbb{C}^{m \times N}$: sampling matrix, with $\ell^2$-normalized columns;
\item $y \in \mathbb{C}^m$: vector of samples;
\item $w \in \mathbb{R}^N$: weights;
\item $\lambda \geq 0$: regularization parameter;
\item $K \in \mathbb{N}$: number of iterations.
\end{itemize}
\textbf{Procedure:}
\begin{enumerate}
\item Let $\hat{x}_0 = 0$ and $S_0 =  \emptyset$;
\item For $k = 1,\ldots,K$:
\begin{enumerate}
\item Find  $\displaystyle j_k  \in \arg\max_{j \in [N]} \Delta_\lambda(x_{k-1},S_{k-1},j)$, with $\Delta_\lambda$ as in \eqref{eq:defDelta};
\item Define $S_k   = S_{k-1} \cup \{j_k\}$;
\item Compute $\displaystyle\hat{x}_{k} \in \arg\min_{v \in \mathbb{C}^{N}} \|A v - y\|_2 \text{ s.t. } \supp(v) \subseteq S_k$.
\end{enumerate}
\end{enumerate}
\textbf{Output:}
\begin{itemize}
\item $\hat{x}_K \in \mathbb{C}^N$: approximate solution to $Az = y$.
\end{itemize}

\end{algorithm}
\begin{remark}
\label{rem:normaliz}
The $\ell^2$-normalization of the columns of $A$ is a necessary condition to apply Algorithm~\ref{alg:WOMP}. If $A$ does not satisfy this hypothesis, is suffices to apply WOMP to the normalized system $\widetilde{A}z = y$, where $\widetilde{A} = AM^{-1}$ and $M$ is the matrix containing the $\ell^2$ norms of the columns of $A$ on the main diagonal and zeroes elsewhere. The approximate solution $\hat{x}_K$ to $\widetilde{A}z = y$ computed via WOMP is then rescaled as $M \hat{x}_K$, which approximately solves $Az = y$.
\end{remark}

The following proposition justifies the weighted variant of OMP considered in Algorithm~\ref{alg:WOMP}. In order to minimize $G_\lambda$ as much as possible, at each iteration, WOMP selects the index $j$ that maximizes the quantity $\Delta_\lambda(x,S,j)$ defined in \eqref{eq:defDelta}. The following proposition makes the role of the quantity $\Delta_\lambda(x,S,j)$ transparent, generalizing relation \eqref{eq:minG0_t} to the weighted case, under suitable conditions on $A$ and $x$ that are verified at each iteration of Algorithm~\ref{alg:WOMP}.  
\begin{proposition}
\label{prop:minG}
Let $\lambda \geq  0$, $S \subseteq [N]$, $A \in \mathbb{C}^{m \times N}$ with $\ell^2$-normalized columns, and $x \in \mathbb{C}^N$ satisfying 
$$
x \in \arg \min_{z \in \mathbb{C}^N} \|y-Az\|_2 \; \text{ s.t. } \supp(z) \subseteq S.
$$
Then, for every $j \in [N]$, the following holds:
$$
\min_{t \in \mathbb{C}} G_\lambda(x+te_j) = G_\lambda(x) - \Delta_\lambda(x,S,j), 
$$
where $G_\lambda$ is defined as in \eqref{eq:defGlambda},  $\Delta_\lambda : \mathbb{C}^N \times 2^{[N]} \times [N] \to \mathbb{R}$ is defined by
\begin{equation}
\label{eq:defDelta}
\Delta_\lambda(x,S,j) 
:= \begin{cases} 
\max\left\{ |(A^*(y-Ax))_j|^2 - \lambda w_j^2, \; 0\right\} & \text{if } j \notin S \\
\max\left\{\lambda w_j^2 - |x_j|^2, \; 0\right\} & \text{if } j \in S \text{ and } x_j \neq 0\\
0 & \text{if } j \in S \text{ and } x_j = 0.
\end{cases}
\end{equation}
\end{proposition}
\begin{proof}
Throughout the proof, we will denote the residual as $r:=y-Ax$. 

Let us first assume $j \notin S$. In this case, we compute
\begin{align*}
G_\lambda(x+te_j)
& = \|y-A(x+te_j)\|_2^2 + \lambda \|x+te_j\|_{0,w} \\
& = \|r\|_2^2 + \underbrace{|t|^2 - 2 \Re(\bar{t}(A^*r)_j) + \lambda (1-\delta_{t,0}) w_j^2}_{=: h(t)}+ \lambda \|x\|_{0,w},
\end{align*}
where $\delta_{x,y}$ is the Kronecker delta function. In particular, we have
$$
h(t)
= 
\begin{cases}
0 & \text{if } t = 0\\
|t|^2 - 2 \Re(\bar{t}(A^*r)_j) +  \lambda w_j^2  &\text{if } t \in \mathbb{C} \setminus \{0\}.
\end{cases}
$$
Now, if $(A^*r)_j = 0$, then $h(t)$ is minimized for $t = 0$ and $\min_{t \in \mathbb{C}}G(x+te_j) = G(x)$. On the other hand, if $(A^*r)_j \neq  0$, by arguing similarly to \cite[Lemma 3.3]{foucart2013mathematical}, we see that 
$$
\min_{t \in \mathbb{C} \setminus\{0\}} h(t) = -|(A^*r)_j|^2 + \lambda w_j^2,
$$
where the minimum is realized for some $t \in \mathbb{C}$ with $|t| = |(A^*r)_j| \neq 0$. In summary,
\begin{align*}
\min_{t \in \mathbb{C}} h(t)
= \min\left\{-|(A^*r)_j|^2 + \lambda w_j^2, 0 \right\}
= -\max\left\{|(A^*r)_j|^2 - \lambda w_j^2, 0\right\},
\end{align*}
which concludes the case $j \notin S$.

Now, assume $j \in S$. Since the vector $x_S  = x|_{S}\in \mathbb{C}^{|S|}$ is a least-squares solution to $A_S z = y$, it  satisfies  $A_S^*(y-A_S x_S) = 0$ and, in particular, $(A^*r)_j = 0$. (Here, $A_S \in \mathbb{C}^{m \times|S|}$ denotes the submatrix of $A$ corresponding to the columns in $S$). Therefore, arguing similarly as before, we have
$$
G(x+te_j) 
= \|r\|_2^2 + \underbrace{|t|^2  + \lambda (1-\delta_{t,-x_j})w_j^2}_{=:\ell(t)} + \lambda \|x-x_je_j\|_{0,w}.
$$
Considering only the terms depending on $t$, it is not difficult to see that
$$
\min_{t \in \mathbb{C}} \ell(t) 
= \min\{|x_j|^2,\lambda w_j^2\}.
$$
As a consequence, for every $j \in S$, we obtain
\begin{align*}
\min_{t \in \mathbb{C}} G(x+te_j) 
& = \|r\|_2^2 + \lambda \|x-x_je_j\|_{0,w} + \min\{|x_j|^2,\lambda w_j^2\}\\
& = G(x) + \min\{|x_j|^2,\lambda w_j^2\} - \lambda(1-\delta_{x_j,0})w_j^2.
\end{align*}
The results above combined with simple algebraic manipulations lead to the desired result. 
\end{proof}

\section{Numerical results}
\label{sec:numerics}


In this section, we show the effectiveness of WOMP (Algorithm~\ref{alg:WOMP}) in the sparse multivariate function approximation setting described in \S\ref{sec:fun_approx}. In particular, we choose the weights $w$ as in \eqref{eq:weights}. We consider the function
\begin{equation}
\label{eq:def_f}
f(t) = \ln\left(d+1+\sum_{k = 1}^d t_k\right), 
\quad \text{with } d = 10. 
\end{equation}
We let $\{\phi_j\}_{j \in \mathbb{N}_0^d}$ be the Legendre and Chebyshev bases and $\nu$ be the respective orthogonality measure. In Fig.~\ref{fig:it_vs_err_legendre} and \ref{fig:it_vs_err_chebyshev} we show the relative $L^2_\nu(D)$-error of the approximate solution $\hat{x}_K$ computed via WOMP as a function of  iteration $K$, for different values of the regularization parameter $\lambda$ in order to solve the linear system $A z = y$, where $A$ and $y$ are defined by \eqref{eq:def_A_y} and where the $\ell^2$-normalization of the columns of $A$ is taken into account according to Remark~\ref{rem:normaliz}. We consider $\lambda = 0$ (corresponding to  OMP) and $\lambda = 10^{-k}$, with $k = 3,3.5,4,4.5,5$. Here, $\Lambda$ is the hyperbolic cross of order $s = 10$, corresponding to $N =|\Lambda| = 571$. Moreover, we consider $m = 60$ and $m = 80$. The results are averaged over $25$ runs and the $L^2_\nu(D)$-error is computed with respect to a reference solution approximated via least squares and using $20 N = 11420$ random i.i.d.\ samples according to $\nu$. We compare the WOMP accuracy with the accuracy obtained via the QCBP program \eqref{eq:WQCBP} with $\eta = 0$ and WQCBP with tolerance parameter $\eta = 10^{-8}$. To solve these two programs we use CVX Version 1.2, a package for specifying and solving convex programs \cite{gb08,cvx}. In CVX, we use the  solver \texttt{'mosek'} and we set CVX precision to  \texttt{'high'}.
\begin{figure}[t]
\centering
\includegraphics[width = 5cm]{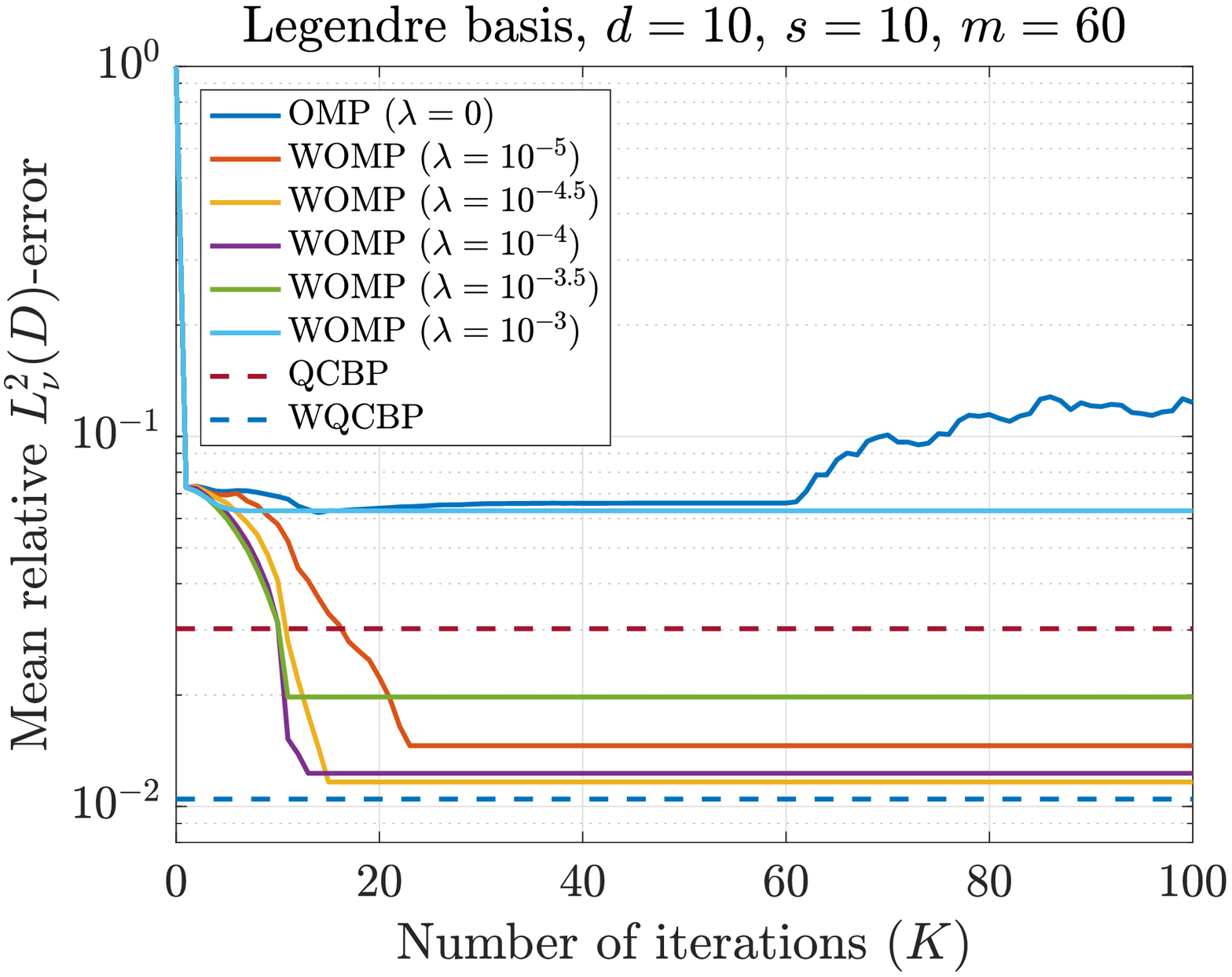}
\includegraphics[width = 5cm]{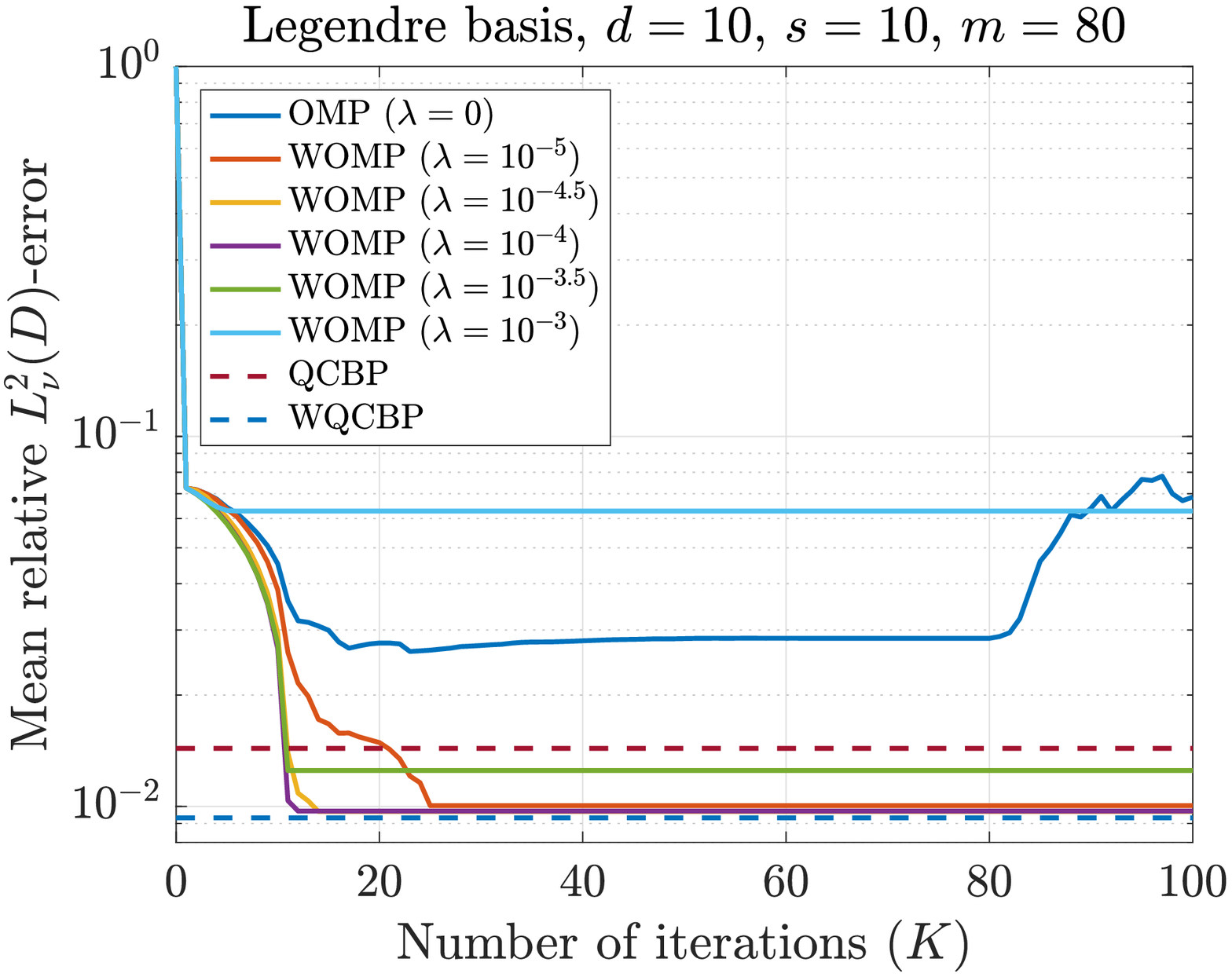}
\caption{\label{fig:it_vs_err_legendre} Plot of the mean relative $L^2_\nu(D)$-error as a function of the number of iterations $K$ of WOMP (Algorithm~\ref{alg:WOMP}) for different values of the regularization parameter $\lambda$ for the approximation of the function $f$ defined in \eqref{eq:def_f} and using Legendre polynomials. The accuracy of WOMP is compared with those of QCBP and WQCBP.}
\end{figure}
\begin{figure}[t]
\centering
\includegraphics[width = 5cm]{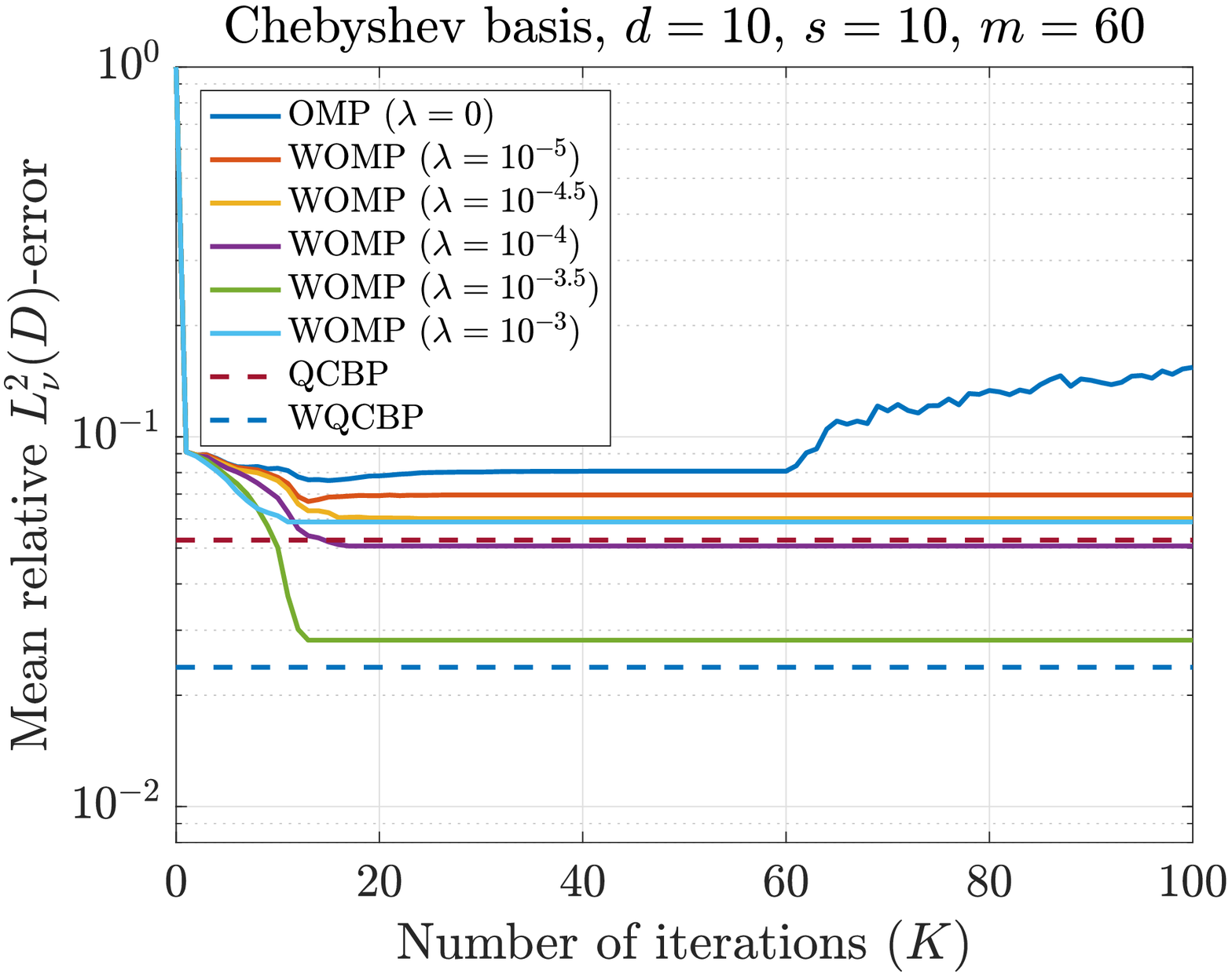}
\includegraphics[width = 5cm]{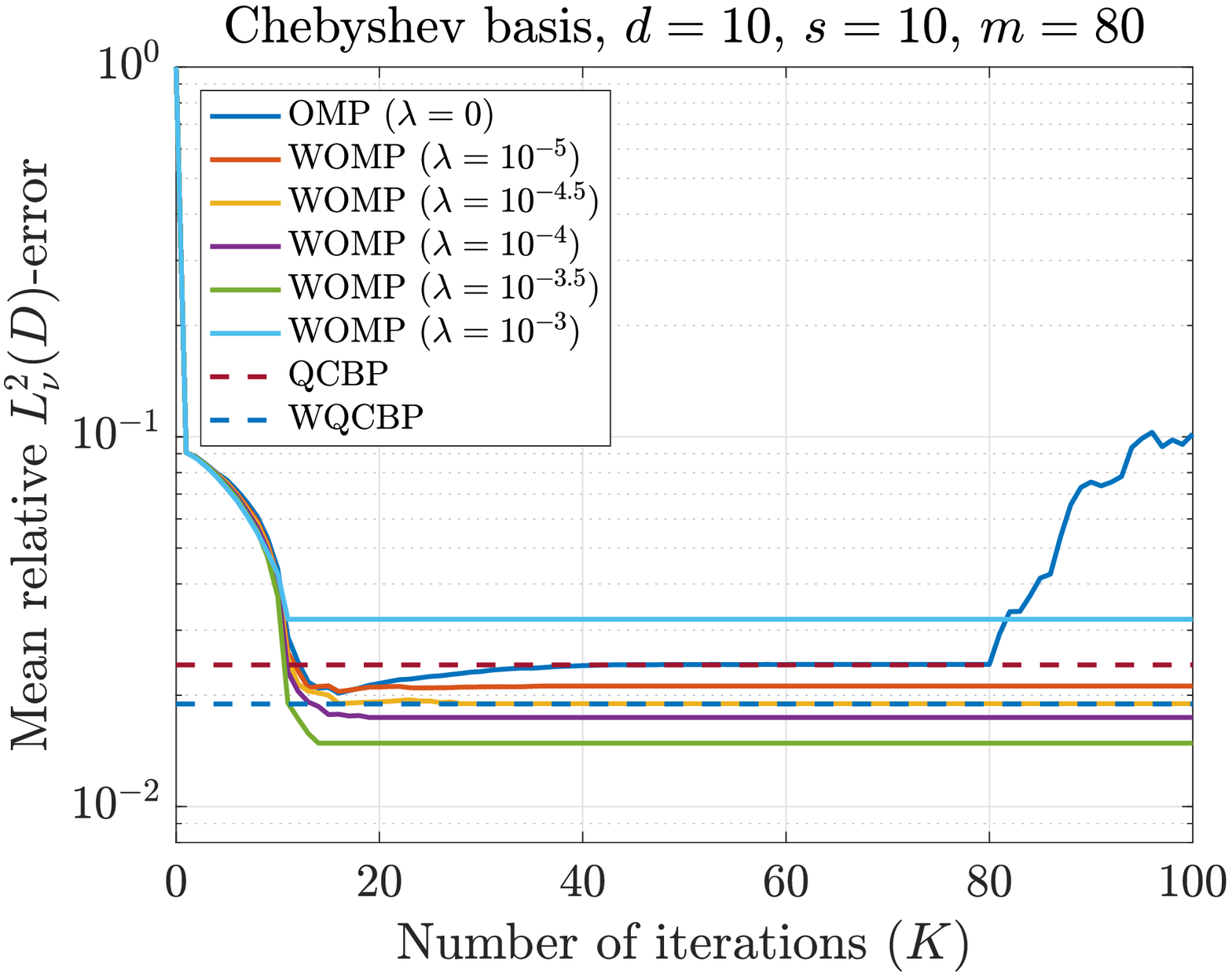}
\caption{\label{fig:it_vs_err_chebyshev} The same experiment as in Fig.~\ref{fig:it_vs_err_legendre}, with Chebyshev polynomials.}
\end{figure}

Fig.~\ref{fig:it_vs_err_legendre} and \ref{fig:it_vs_err_chebyshev} show the benefits of using weights as compared to the unweighted OMP approach, when the parameter $\lambda$ is tuned appropriately. A good choice of $\lambda$ for the setting considered here seem to be between $10^{-4.5}$ and $10^{-3.5}$. We also observe that WOMP is able to reach similar level of accuracy as WQCBP. An interesting feature of WOMP with respect to OMP is its better stability. We observe than after  the $m$-th iteration, the OMP accuracy starts getting substantially worse. This can be explained by the fact that when $K$ approaches $N$, OMP tends to destroy sparsity by fitting the data too much. This phenomenon is not observed in WOMP, thanks to its ability to keep the support of $\hat{x}_k$ small via the explicit enforcement of the wighted sparsity prior (see Fig.~\ref{fig:it_vs_size}).

\begin{figure}[t]
\centering
\includegraphics[width = 5cm]{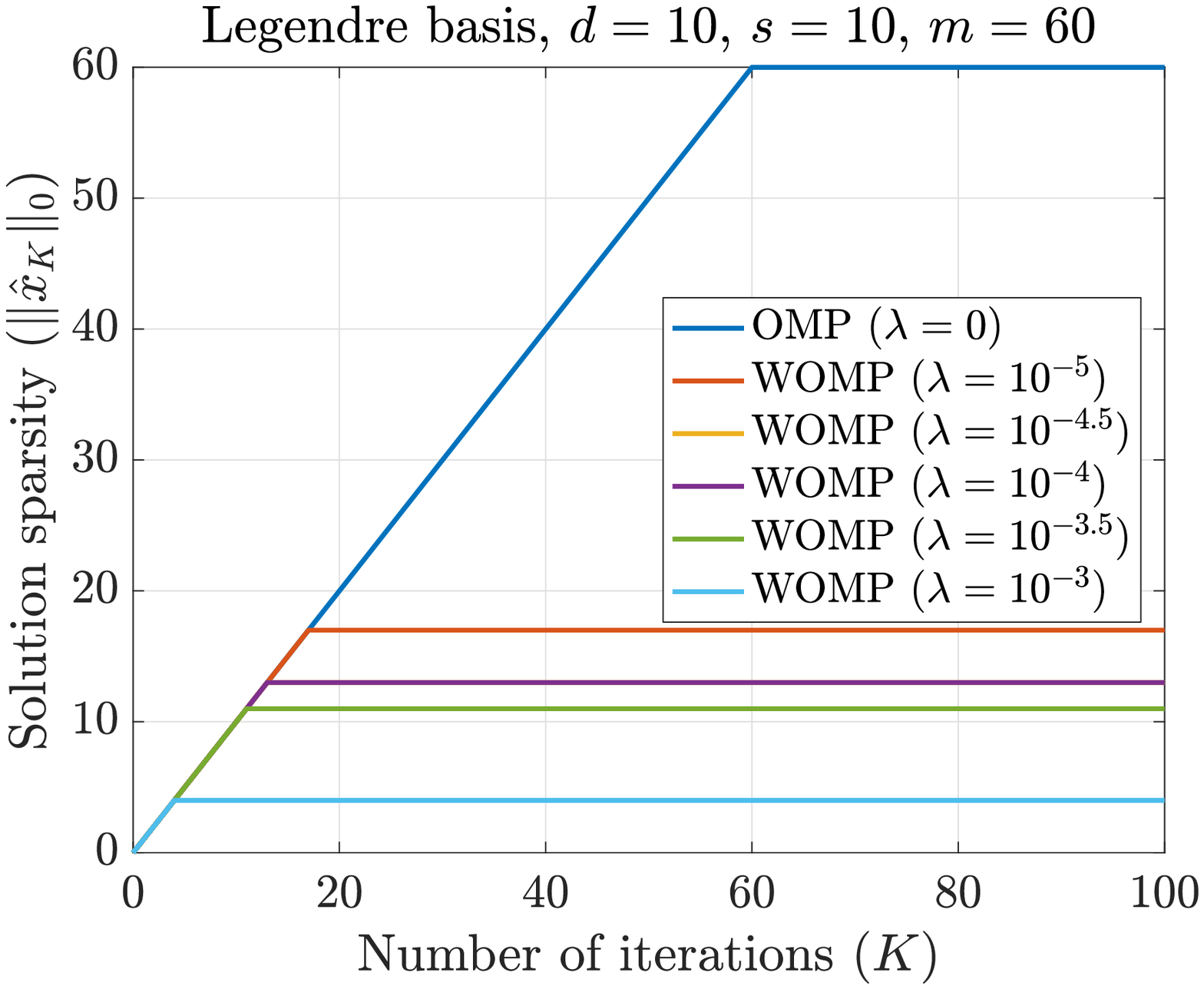}
\includegraphics[width = 5cm]{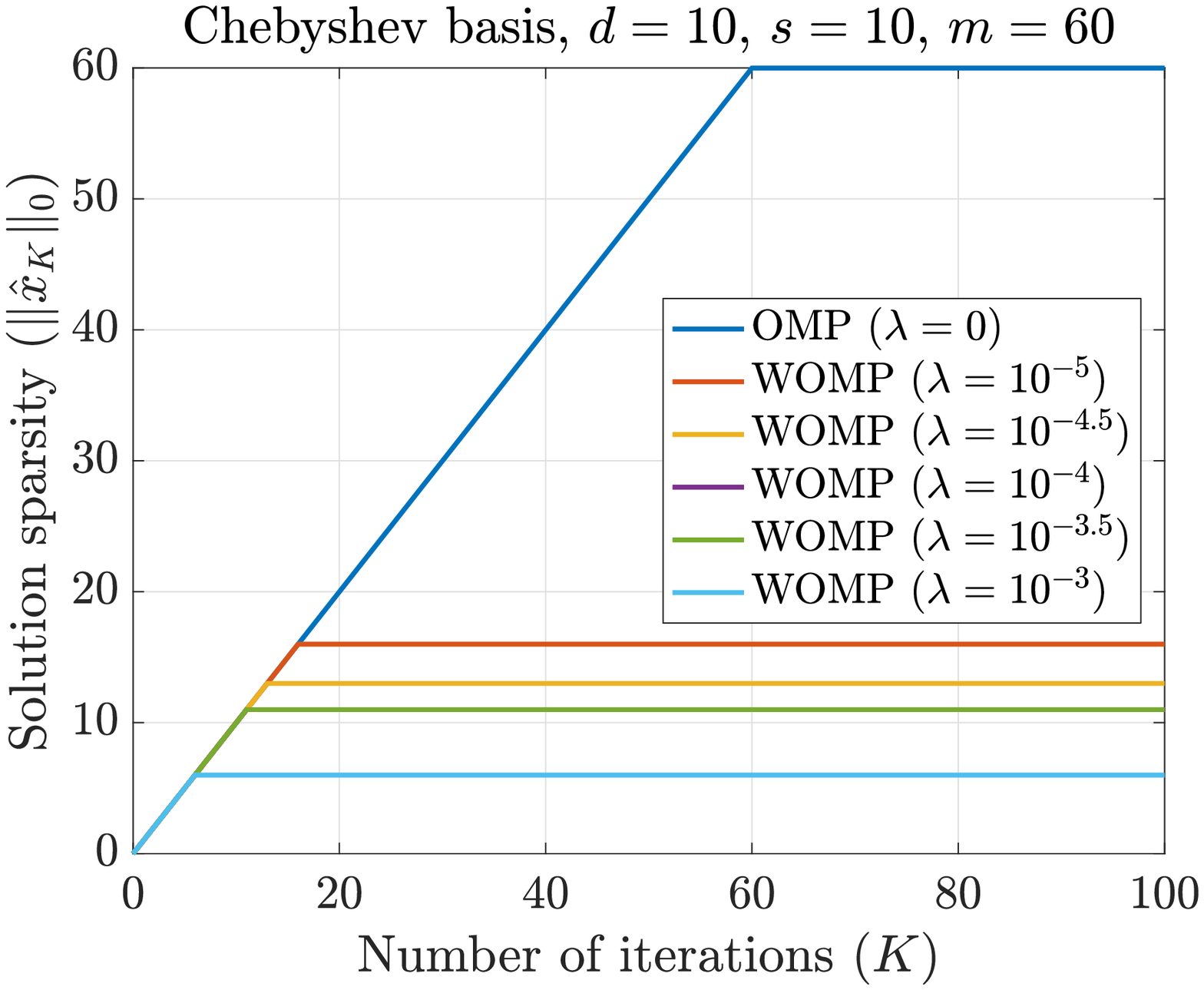}
\caption{\label{fig:it_vs_size}Plot of the support size of $\hat{x}_K$ as a function of the number of iterations $K$ for WOMP in the same setting as in Fig.~\ref{fig:it_vs_err_legendre} and \ref{fig:it_vs_err_chebyshev}, with Legendre (left) and Chebyshev (right) polynomials. The larger the regularization parameter $\lambda$, the sparser solution. (In the left plot, the curves relative to $\lambda = 10^{-4.5}$ and $\lambda = 10^{-4}$ overlap. In the right plot, the same happens for $\lambda = 10^{-4}$ and $\lambda = 10^{-3.5}$.)}
\end{figure}

We show the better computational efficiency of WOMP with respect to the convex minimization programs QCBP and WQCBP solved via CVX by tracking the runtimes for the different approaches. In Table~\ref{tab:runtimes_1} we show the running times for the different recovery strategies. 
\begin{table}[t]
\centering
\begin{tabular}{c|c|c|c|c|ccccc}
\multirow{2}{*}{Basis} & \multirow{2}{*}{$m$} & \multirow{2}{*}{QCBP} & \multirow{2}{*}{WQCBP} & \multirow{2}{*}{OMP} & \multicolumn{5}{c}{WOMP with $\lambda$ as below} \\
 & & &  & & $10^{-5}$ & $10^{-4.5}$ & $10^{-4}$ & $10^{-3.5}$ & $10^{-3}$ \\
\hline
Legendre & 60 & 1.9e-01 & 2.0e-01 & 1.6e-02 & 1.3e-02 & 1.2e-02 & 1.3e-02 & 1.2e-02 & 1.2e-02 \\ 
Legendre & 80 & 2.1e-01 & 2.1e-01 & 1.7e-02 & 1.5e-02 & 1.3e-02 & 1.4e-02 & 1.4e-02 & 1.3e-02 \\ 
Chebyshev & 60 & 1.9e-01 & 1.9e-01 & 1.5e-02 & 1.3e-02 & 1.2e-02 & 1.2e-02 & 1.2e-02 & 1.2e-02 \\ 
Chebyshev & 80 & 2.1e-01 & 2.1e-01 & 1.7e-02 & 1.5e-02 & 1.3e-02 & 1.4e-02 & 1.4e-02 & 1.4e-02 \\  
\end{tabular}
\caption{\label{tab:runtimes_1}Comparison of the computing times for WQCBP and $K = 25$ iterations of WOMP.}
\end{table}
The running times for WOMP are referred to $K=25$ iterations, sufficient to reach the best accuracy for every value of $\lambda$ as shown in Fig.~\ref{fig:it_vs_err_legendre} and \ref{fig:it_vs_err_chebyshev}. Moreover, the computational times for WOMP take into account the $\ell^2$-normalization of the columns of $A$ (see Remark~\ref{rem:normaliz}). WOMP consistently outperforms convex minimization, being more than ten times faster in all cases. We note that in this comparison a key role is played by the parameter $K$ or, equivalently, by the sparsity of the solution. Indeed, in this case, considering a larger value of $K$ would result is a slower performance of WOMP, but it would not improve the accuracy of the WOMP solution (see Fig.~\ref{fig:it_vs_err_legendre} and \ref{fig:it_vs_err_chebyshev}).

\section{Conclusions}
\label{sec:conclusions}

We have considered a greedy recovery strategy for high-dimensional function approximation from a small set of pointwise samples. In particular, we have proposed a generalization of the OMP algorithm to the setting of weighted sparsity (Algorithm~\ref{alg:WOMP}). The corresponding greedy selection strategy is derived in Proposition~\ref{prop:minG}. 

Numerical experiments show that WOMP is an effective strategy for high-dimensional approximation, able to reach the same accuracy level of WQCBP while being considerably faster when the target sparsity level is small enough.  A key role is played by the regularization parameter $\lambda$, which may be difficult to tune due to its sensitivity to the parameters of the problem ($m$, $s$, and $d$), and on the polynomial basis employed. In other applications, where explicit formulas for the weights as \eqref{eq:weights} are not available, there might also be a nontrivial interplay between $\lambda$ and $w$. In summary, despite the promising nature of the numerical experiments illustrated in this paper, a more extensive numerical investigation is needed in order to study the sensitivity of WOMP with respect to $\lambda$. Moreover, a theoretical analysis of the WOMP approach might highlight practical recipe for the choice of this parameter, similarly to \cite{adcock2017correcting}. This type of analysis may also help identifying the sparsity regime where WOMP outperforms weighted $\ell^1$ minimization, which, in turn, could be formulated in terms of suitable assumptions on the regularity of $f$. These questions are beyond the scope of this paper and will be object of future work.


\section*{Acknowledgements}

The authors acknowledge the support of the Natural Sciences and Engineering Research Council of Canada through grant number 611675, and of the Pacific Institute for the Mathematical Sciences (PIMS) Collaborative Research Group  ``High-Dimensional Data Analysis''. S.B.\ also acknowledges the support of the PIMS Postdoctoral Training Centre in Stochastics. 

\bibliographystyle{plain}
\bibliography{biblio}

\end{document}